\documentclass[]{amsart}

\usepackage{amssymb}
\usepackage{enumerate}
\usepackage{graphicx} 

\usepackage[all]{xy}

\newtheorem{proposition}{Proposition}
\newtheorem{theorem}[proposition]{Theorem}
\newtheorem{lemma}[proposition]{Lemma}
\newtheorem{corollary}[proposition]{Corollary}

\theoremstyle{definition}  
\newtheorem{axiom}[proposition]{Axiom} 
\newtheorem{definition}[proposition]{Definition}
\newtheorem{example}[proposition]{Example}

\newcommand{\thmref}[1]{Theorem~\ref{#1}}
\newcommand{\propref}[1]{Proposition~\ref{#1}}

\newcommand{\lemref}[1]{Lemma~\ref{#1}}
\newcommand{\cororef}[1]{Corollary~\ref{#1}}
\newcommand{\defref}[1]{Definition~\ref{#1}}
\newcommand{\examref}[1]{Example~\ref{#1}}

\newcommand{\cM}{\mathcal{M}}
\newcommand{\Top}{\mathrm{\bf Top^w}}
\newcommand{\SX}{\Sigma X}
\newcommand{\CP}{{\mathbb C}P}
\newcommand{\secat}{{\rm secat}\,}
\newcommand{\cat}{{\rm cat}\,}
\newcommand{\Relcat}{{\rm Relcat}\,}
\newcommand{\relcat}{{\rm relcat}\,}
\newcommand{\Cat}{{\rm Cat}\,}
\newcommand{\Pushcat}{{\rm Pushcat}\,}

\newcommand{\ipu}{{i+1}}
\newcommand{\compl}{{\rm compl}\,}
\newcommand{\relcompl}{{\rm relcompl}\,}
\newcommand{\Pushcompl}{{\rm Pushcompl}\,}
\newcommand{\Compl}{{\rm Compl}\,}
\newcommand{\id}{{\rm id}}
\newcommand{\inc}{{\rm in}}
\newcommand{\pro}{{\rm pr}}
\newcommand{\resp}{respectively: }
\newcommand{\undemi}{\frac{1}{2}}
\renewcommand{\leq}{\leqslant}
\renewcommand{\geq}{\geqslant}
\newcommand{\equi}{\mathrel{\scalebox{1.2}{$\sim$}}}

\title[Approximations of secat and topological complexity]{Up-to-one approximations of sectional category and topological complexity}

\author{Jean-Paul Doeraene and Mohammed El Haouari}



\begin{document}


\begin{abstract}
James' sectional category and Farber's topological complexity are studied in a general and unified framework.

We introduce `relative' and `strong relative' forms of the category for a map 
and show that both can differ from sectional category by just one. 
A map has sectional or relative category at most $n$ if, and only if, it is `dominated' in a (different) sense by a map with strong relative category at most $n$.
A homotopy pushout can increase sectional category but neither homotopy pushouts, nor homotopy pullbacks, can increase (strong) relative category. This makes (strong) relative category a convenient tool to study sectional category. We completely determine the sectional and relative categories of the fibres of the Ganea fibrations. 

In particular, the `topological complexity' of a space is the sectional category of the diagonal map, and so it can differ from the (strong) relative category of the diagonal by just one. We call the strong relative category of the diagonal `strong complexity'. We show that the strong complexity of a suspension is at most two.
\end{abstract}
 
\maketitle 


Our aim is to build a general and unified framework to study James' \emph{sectional category} of a map \cite{Jam78}, 
and Farber's \emph{topological complexity} of a space, which is the sectional category of the diagonal map \cite{Far03}.

In the first section, we give definitions of secat in the `Ganea-Whitehead style', and introduce variants of secat, called \emph{relative category}
and \emph{strong relative category}. The main result of this section, \thmref{takens}, is that all these variants differ  by just one. 
Also there is a kind of `attachment formula' for relative category, which is \propref{relcatcone}.
These results, and others, come essentially from Lemmas \ref{cylsomme} and \ref{jointmagique} which assert that
homotopy pushouts and homotopy pullbacks do not increase (strong) relative
category. It is fruitful to jointly consider sectional and relative categories, because they
do not share the same properties: see \lemref{cylsomme} or \propref{secatbut} for instance.
At the end of the section, we show in \cororef{secatfibreganea} that the sectional category of the fibre $F_i \to G_i$ of the $i^{\text{\small th}}$ Ganea fibration $G_i \to X$ is $\min\{i,\cat(X)\}$, while its relative category is  $\min\{i+1,\cat(X)\}$. Actually, this result comes as a particular case of the determination of the sectional category and relative category of maps of some `relative' Ganea construction, which is given in \thmref{secatganea}.

In the second section, we apply the results of the first section to complexity. Variants of complexity, corresponding to variants of sectional category, differ by just one.  In particular, the strong relative category of the diagonal $\Delta\colon X \to X\times X$ is called the \emph{strong complexity} of $X$, and in \thmref{complsusp} we show, by an explicit computation, that the strong complexity of any suspension is at most two. 

We work in the category of well-pointed topological spaces $\Top$ ({\em well-pointed} means that the inclusion of the base point is a closed cofibration) \cite{Str72}.
But we don't use any construction particular to topological spaces, we use only homotopy pullbacks and homotopy pushouts;
so our techniques also apply in algebraic categories used to model topological spaces (commutative differential graded algebras, modules over a d.g.a., etc.).
More precisely, our results apply in the general context of a closed model category $\cM$, satisfying the {\em Cube axiom} (see the appendix for details).  
However, in these categories, in order to use the usual property of `homotopy', one needs cofibrant-fibrant objects, but it's possible to circumvent this difficulty with some technical complications; see \cite{Doe93} or \cite{DiaCal12}.

We thank Professor Peter Landweber for his careful reading and useful suggestions. We also thank the referee for his valuable remarks.

\section[Sectional category]{Sectional category}\label{sectionwhiteheadganea}

\subsection{The Ganea construction}

\begin{definition}\label{ganea}
For any map $\iota_X\colon A \to X$ of $\cM$,
the \emph{Ganea construction} of $\iota_X$
is the following sequence of homotopy commutative diagrams ($i \geq 0$):
$$\xymatrix{
&A\ar[dr]_{\alpha_{i+1}}\ar[rrrd]^{\iota_X}\\
F_{i}\ar[rd]_{\beta_i}\ar[ur]^{\eta_{i}}&&G_{i+1}\ar[rr]|-(.35){g_{i+1}}&&X\\
&G_{i}\ar[ru]^{\gamma_{i}}\ar[rrru]_{g_{i}}}$$
where the outside square is a homotopy pullback, the inside
square is a homotopy pushout and the map $g_{i+1} \colon G_{i+1} \to X$
is the whisker map induced by this homotopy pushout.
The induction starts with  $g_0 = \iota_X \colon A \to X$.
\end{definition}

We  denote $G_i$ by $G_i(\iota_X)$, or by $G_i(X,A)$.
If $\cM$ is pointed with $*$ as zero object, we write $G_i(X)=G_i(X,*)$.

The sequence of homotopy commutative diagrams above extends to:
$$\xymatrix{
&&&A\ar[dr]_{\alpha_{i+1}}\ar[rrrd]^{\iota_X}\\
A\ar@{-->}[rr]|-(.7){\theta_i}\ar@{=}[rrru]\ar[rrrd]_{\alpha_{i}}&&F_i\ar[rd]^{\beta_i}\ar[ur]&&G_{i+1}\ar[rr]|-(.35){g_{i+1}}&&X\\
&&&G_i\ar[ur]^{\gamma_i}\ar[ru]\ar[rrru]_{g_{i}}
}$$
where $\alpha_0 = \id_A$ and, since $g_i \circ \alpha_i \simeq \iota_X$, the map $\theta_i\colon A \to F_i$ is the whisker map induced by the homotopy pullback $F_i$. 
Notice also that $\gamma_i \circ \alpha_i \simeq \alpha_{i+1}$.

\begin{definition}\label{LSganea}
Let $\iota_X\colon A \to X$ be a map of $\cM$. 

1) The \emph{sectional category} of $\iota_X$ is the least integer $n$ such that the map $g_n\colon G_n(\iota_X)\to X$ has a homotopy section, i.e. there exists a map $\sigma\colon X \to G_n(\iota_X)$ such that $g_n \circ \sigma \simeq \id_X$.

2) The \emph{relative category} of $\iota_X$ is the least integer $n$ such that the map $g_n\colon G_n(\iota_X)\to X$ has a homotopy section $\sigma$ and $\sigma \circ \iota_X \simeq \alpha_n$.
\end{definition}

We denote the sectional category  by $\secat(\iota_X)$ or $\secat(X,A)$, and the relative category by $\relcat(\iota_X)$ or $\relcat(X,A)$.
If $\cM$ is pointed with $\ast$ as zero object, we write $\cat(X)=\secat(X,\ast) = \relcat(X,\ast)$.

\smallskip
In $\Top$, $\cat(X)$ is T.~Ganea's version of the {\em category} of L.~Lusternik and L.~Schnirelmann. 
For a normal path-connected space $X$ with non-degenerate basepoint,
the definition here is equivalent to the original one with open covers of $X$ (up to a shift by 1): 
${\rm LS-}\cat(X) = \cat(X) + 1$.

For a comprehensive review on category and
sectional category of topological spaces and maps, see \cite{Jam78};
these notions are also deeply analysed in \cite{CLOT03}.
Warning: Our relative category is \emph{not} that of E.~Fadell \cite{Fad96}.


\subsection{Strong pushout category}

Here we define the `strong pushout category' of a map and establish its basic properties.
The principal result here is \propref{catretractstrong}, which characterizes sectional and relative category in terms of this invariant.

\begin{definition}\label{strongsecat}
The \emph{(strong) pushout category of a  map $\iota_X\colon A\to X$ of $\cM$} is the least integer $n$ such that:
\begin{itemize}
\item There are maps $\iota_0\colon A \to X_0$ and a homotopy inverse $\lambda\colon X_0\to A$, i.e.  $\iota_0
  \circ \lambda \simeq \id_{X_0}$ and $\lambda \circ \iota_0 \simeq \id_A$;

\item for each $i$, $0\leq i < n$, there exists a homotopy pushout
$$\xymatrix@C=3pc{
Z_i\ar[r]\ar[d]&A\ar[d]^{\iota_{i+1}}\\
X_i\ar[r]_{\chi_i}&X_{i+1}
}$$
such that $\chi_i\circ \iota_i \simeq \iota_{i+1}$;

\item $X_n=X$ and   $\iota_n \simeq \iota_X$.
\end{itemize}
\end{definition}

We denote the strong pushout category by $\Pushcat(\iota_X)$, or $\Pushcat(X,A)$.

\smallskip
In particular, $\Pushcat(\iota_X) = 0$ iff $\iota_X$ is a homotopy equivalence.
When this is not true, then $\Pushcat(\iota_X) = 1$ if there is a homotopy pushout:
$$\xymatrix{
Z\ar[r]^{\rho}\ar[d]_{\rho'}&A\ar[d]^{\iota_X}\\
A\ar[r]_{\iota_X}&X
}$$
(The maps $\rho$ and $\rho'$ are {\em not} necessarily homotopic, unless $\iota_X$ is a homotopy monomorphism.)


\smallskip
For $0 \leq i < n$, define the sequence of maps $\xi_{i} = \xi_{i+1} \circ
\chi_{i}$  beginning with $\xi_n = \id_X$. 
We have $\xi_n \circ \iota_n \simeq \iota_X$ and $\xi_{i}\circ \iota_{i} = \xi_{i+1} \circ \chi_{i} \circ \iota_{i} \simeq \xi_{i+1} \circ
\iota_{i+1} \simeq \iota_X$ by decreasing induction. So we have a sequence of homotopy commutative diagrams:
$$\xymatrix{
&A\ar[drrr]^{\iota_X}\ar[dr]_{\iota_{i+1}}\\
Z_{i}\ar[ru]\ar[rd]&&X_{i+1}\ar[rr]|-(.35){\xi_{i+1}}&&X\\
&X_{i}\ar[ru]^{\chi_{i}}\ar[rrru]_{\xi_{i}}&&&}$$
where the inside square is a homotopy pushout and the map $\xi_{i+1}$ is the
whisker map induced by this homotopy pushout. Notice that $\xi_0 \simeq \xi_0 \circ \iota_0
\circ \lambda \simeq \iota_X \circ \lambda$.

\smallskip
Observe that, clearly, $\Pushcat(\iota_i) \leq i$. Also we have:

\begin{lemma}\label{pushcone}
Let $\iota_X\colon A\to X$ be a map of $\cM$. Suppose we have a homotopy pushout:
$$\xymatrix{
Z\ar[r]^{\rho}\ar[d]&A\ar[d]^{\iota_C}\\
X\ar[r]_\chi&C
}$$
where $\chi\circ\iota_X\simeq\iota_C$.
Then, we have 
$\Pushcat(\iota_C)\leq \Pushcat(\iota_X)+1$.
\end{lemma}

\begin{definition}\label{dominated}
Consider the following diagram
$$\xymatrix@C=1.5pc@R=1.5pc{
&&X\ar@/^1pc/[dd]^{\varphi}\\
A\ar[rru]^{\iota_X}\ar[drr]_{\iota_Y}&&\\
&&Y\ar@{-->}@/^1pc/[uu]^{\sigma}
}$$
such that 
$\varphi\circ \iota_X\simeq \iota_Y$.

1) If $\sigma$ is a homotopy section of $\varphi$, i.e. $\varphi\circ \sigma\simeq \id_{Y}$, we say that $\iota_Y$ is \emph{(simply) dominated by $\iota_X$.}

2)  If $\sigma$ is a homotopy section of $\varphi$ and $\sigma \circ \iota_Y \simeq \iota_X$,  we say that $\iota_Y$ is \emph{relatively dominated by $\iota_X$.}

3) If  $\sigma$ is a homotopy inverse of $\varphi$, i.e. $\varphi\circ \sigma\simeq \id_{Y}$ and $\sigma\circ \varphi\simeq \id_X$, 
we say that $\iota_Y$ and $\iota_X$ are of \emph{the same homotopy type}.
Notice that in this case, we have also $\sigma \circ \iota_Y \simeq \iota_X$.
\end{definition}

Our definitions of $\secat(\iota_X)$, $\relcat(\iota_X)$, $\Pushcat(\iota_X)$, and $\Relcat(\iota_X)$ below, are designed in such a way that they become \emph{de facto} invariants of  homotopy type.

\begin{proposition}\label{catretractstrong}
Let $\iota_Y\colon A\to Y$ be a map of $\cM$. The following conditions are equivalent:
\begin{enumerate}
\item $\secat(\iota_Y)\leq n$ (\resp $\relcat(\iota_Y)\leq n$);
\item the map $\iota_Y$ is simply (\resp relatively) dominated by a map $\iota_X\colon A\to X$ (see \defref{dominated} above) such that $\Pushcat(\iota_X)\leq n$.
\end{enumerate}
\end{proposition}

\begin{proof}
Consider the map $\alpha_n\colon A \to G_n(\iota_Y)$ as in \defref{ganea} and notice that $\Pushcat(\alpha_n) \leq n$.
If  $\secat(\iota_Y)\leq n$, then $\iota_Y$ is simply dominated by $\alpha_n$.
If  $\relcat(\iota_Y)\leq n$, then $\iota_Y$ is relatively dominated by $\alpha_n$.

\smallskip
For the reverse direction, we suppose the existence of maps  as in \defref{dominated} such that
$\varphi\circ \iota_X\simeq \iota_Y$ and $\varphi\circ \sigma\simeq {\rm id}_Y$.
 From $\Pushcat(\iota_X)\leq n$, we get a sequence of homotopy commutative  diagrams, for $0\leq i < n$, as in \defref{strongsecat},
which gives the top part of the following diagram.
 
We show by induction that the map
$\varphi\circ \xi_i\colon X_i\to Y$ factors through $g_i\colon G_i(\iota_Y)\to Y$ up to homotopy.
This is true for $i=0$ since we have $\xi_0 \simeq \iota_X \circ \lambda$,
so $\varphi\circ \xi_0 \simeq \varphi\circ\iota_X \circ \lambda \simeq \iota_Y
\circ \lambda = g_0 \circ \lambda$.
Suppose now that we have a map $\lambda_{i}\colon X_{i}\to G_{i}(\iota_Y)$
such that $g_{i}\circ \lambda_{i}\simeq \varphi\circ \xi_{i}$. Then we construct a homotopy commutative diagram
$$\xymatrix@C=1.5pc@R=1.5pc{
&Z_{i}\ar[ld]\ar[rr]\ar@{-->}[dd]|!{[ld];[rd]}\hole
&&A\ar[ld]_{\iota_{i+1}}\ar[rd]\ar@{=}[dd]|\hole&\\
X_{i}\ar[rr]\ar[dd]_{\lambda_{i}}&&
X_{i+1}\ar@{..>}[dd]_(.3){\lambda_{i+1}}\ar[rr]_(.34){\xi_{i+1}}&&X\ar[dd]^{\varphi}\\
&F_{i}\ar[rr]|(.5)\hole
\ar[ld]&&A\ar[ld]^{\alpha_{i+1}}\ar[rd]&\\
G_{i}(\iota_Y)\ar[rr]&&G_{i+1}(\iota_Y)\ar[rr]_{g_{i+1}}&&Y
}$$
where $\xymatrix@1{Z_{i}\ar@{-->}[r]& F_{i}}$ is  the whisker map induced by the bottom homotopy pullback and 
$\lambda_{i+1}\colon \xymatrix@1{X_{i+1}\ar@{.>}[r]&  G_{i+1}(\iota_Y)}$ is the whisker map induced by the top homotopy pushout.
The composite $g_{i+1}\circ \lambda_{i+1}$ is homotopic to $\varphi\circ \xi_{i+1}$ by \lemref{whiskercube}.
Hence the inductive step is proven.

At the end of the induction, we have $g_n\circ \lambda_n\simeq \varphi\circ \xi_n = \varphi \circ \id_X = \varphi$. As we have a homotopy section $\sigma\colon Y\to X_n=X$ of $\varphi$, we get a homotopy section $\lambda_n\circ \sigma$ of $g_n$.
If, in addition, $\sigma \circ \iota_Y \simeq \iota_X$, then $\lambda_n \circ \sigma \circ \iota_Y \simeq \lambda_n \circ \iota_X \simeq \lambda_n \circ \iota_n \simeq \alpha_n$.
\end{proof}

\begin{corollary}\label{catdominated}
Let $\iota_Y\colon A\to Y$ be simply (\resp relatively) dominated by $\iota_X\colon A\to X$.
Then we have $\secat(\iota_Y)\leq \secat(\iota_X)$ (\resp $\relcat(\iota_Y)\leq \relcat(\iota_X)$).
\end{corollary}


\subsection{Strong relative category}

Here we define the `strong relative category' of a map and establish its basic properties.
The principal results here are Lemmas \ref{cylsomme} and \ref{jointmagique} which assert that
homotopy pushouts and homotopy pullbacks do not increase (strong) relative category.

\begin{definition}\label{strongrelcat}
The \emph{strong relative category of a  map $\iota_X\colon A\to X$ of $\cM$} is the least integer $n$ such that:
\begin{itemize}
\item There are maps $\iota_0\colon A \to X_0$ and a homotopy inverse $\lambda\colon X_0\to A$, i.e.  $\iota_0
  \circ \lambda \simeq \id_{X_0}$ and $\lambda \circ \iota_0 \simeq \id_A$;
\item for each $i$, $0\leq i < n$, there exists a homotopy commutative diagram
$$\xymatrix{
&&&A\ar[dr]^{\iota_{i+1}}\\
A\ar@{-->}[rr]_(.7){\sigma_i}\ar@{=}[rrru]\ar[rrrd]_{\iota_i}&&Z_i\ar[rd]\ar[ur]_{\rho_i}&&X_{i+1}\\
&&&X_i\ar[ur]_{\chi_i}
}$$
where the inside square is a homotopy pushout;
\item $X_n=X$ and   $\iota_n \simeq \iota_X$.
\end{itemize}

\end{definition}

We denote the strong relative category by $\Relcat(\iota_X)$, or $\Relcat(X,A)$.

If $\cM$ is pointed with $*$ as zero object, we write $\Cat(X)=\Relcat(X,\ast) = \Pushcat(X,\ast)$.  

\smallskip
In $\Top$, $\Cat(X)$ is the homotopy invariant version of the R.~Fox's {\em strong category}, see \cite{Gan67}. 

\smallskip
In particular, $\Relcat(\iota_X) = 0$ iff $\iota_X$ is a homotopy equivalence.
When this is not true, then $\Relcat(\iota_X) = 1$ if there is a homotopy pushout:
$$\xymatrix{
Z\ar[r]^{\rho}\ar[d]_{\rho'}&A\ar[d]^{\iota_X}\\
A\ar[r]_{\iota_X}&X
}$$
such that $\rho$ and $\rho'$ have a common homotopy section $\sigma$.
The following proposition shows that this situation occurs for any homotopy cofibre:

\begin{proposition}\label{cylcofibre}
Assume $\cM$ is pointed. Given a homotopy cofibration sequence $\xymatrix{Y\ar[r]|f&A\ar[r]|{\iota_X}&X}$, we have $\Relcat(\iota_X) \leq 1$.
\end{proposition}

\begin{proof}
To get a homotopy pushout as above, use \propref{doublecofibre} and let $Z = A \vee Y$, $\rho \simeq (\id_A,f)$ is the whisker map of $\id_A$ and $f$,
and $\rho' \simeq \pro_1 \colon A \vee Y\to A$ and $\sigma \simeq \inc_1\colon A \to A \vee Y$ are the obvious maps.
\end{proof}

In section \ref{CompSX}, we will construct a two-steps example showing that the diagonal $\Delta \colon \SX \to \SX \times \SX$ of a suspension has $\Relcat(\Delta) \leq 2$.

\smallskip
Observe that for the map $\alpha_i\colon A \to G_i(\iota_X)$ of the Ganea construction, we have $\Relcat(\alpha_i) \leq i$.
And since, clearly, $\Pushcat(\iota_X) \leq \Relcat(\iota_X)$, \propref{catretractstrong} can be extended to the following:

\begin{proposition}\label{catretractcone}
Let $\iota_Y\colon A\to Y$ be a map of $\cM$. The following conditions are equivalent:
\begin{enumerate}
\item $\secat(\iota_Y)\leq n$ (\resp $\relcat(\iota_Y)\leq n$);
\item the map $\iota_Y$ is simply (\resp relatively) dominated by a map $\iota_X\colon A\to X$ such that $\Pushcat(\iota_X)\leq n$;
\item the map $\iota_Y$ is simply (\resp relatively) dominated by a map $\iota_{X'}\colon A\to X'$ such that $\Relcat(\iota_{X'})\leq n$.
\end{enumerate}
\end{proposition}

\begin{lemma}\label{cylsomme}
If $\iota_X \colon A \to X$ and $A \to B$ are maps of $\cM$, consider the homotopy pushout:
$$\xymatrix{
A\ar[d]\ar[r]^{\iota_X}&X\ar[d]\\
B \ar[r]_{\kappa_S}&S
}$$
We have  $\Relcat(\kappa_S) \leq \Relcat(\iota_X)$ and $\relcat(\kappa_S) \leq \relcat(\iota_X)$.
\end{lemma}

\begin{proof}
Let $\Relcat(\iota_X) = n$. Consider the sequence of homotopy pushouts as in
\defref{strongrelcat} and the sequence of maps $\xi_i$ defined after \defref{strongsecat}. This gives the top part of the next diagram. We extend it to a sequence of homotopy commutative diagrams, for $0\leq i< n$,
$$\xymatrix@C=1.5pc@R=1.5pc{
A\ar[rr]^{\sigma_i}\ar[dd]\ar[rd]^(.54){\iota_{i}}&&Z_{i}\ar[rr]^{\rho_i}\ar[ld]\ar[dd]|\hole&&A\ar[rd]^{\iota_X}\ar[ld]_(.54){\iota_{i+1}}\ar[dd]|\hole&\\
&X_{i}\ar[rr]\ar[dd]&&X_{i+1}\ar[dd]\ar[rr]_(.3){\xi_{i+1}}&&X\ar[dd]\\
B\ar[rr]_(.3){\tilde\sigma_i}|(.5)\hole\ar[rd]_{\tilde\iota_{i}}&&R_{i}\ar[ld]\ar[rr]_(.3){\tilde\rho_i}|(.5)\hole&&B\ar[rd]_{\kappa_S}\ar[ld]^{\tilde\iota_{i+1}}&\\
&S_{i}\ar[rr]&&S_{i+1}\ar[rr]&&S
}$$
building all vertical faces as homotopy pushouts.
From the Prism lemma \ref{prisme}, we know that the bottom face of the inside cube is a homotopy pushout
and that  $\tilde\rho_{i} \circ \tilde\sigma_{i} \simeq \id_B$.
Also, since $\iota_0\colon A\to X_0$ has a homotopy inverse, $\tilde\iota_0\colon B \to S_0$ has a homotopy inverse as well.
Finally, since $\xi_n = \id_X$, we may assume $S_n = S$ and $\tilde\iota_n \simeq \kappa_S$. This means $\Relcat(\kappa_S)\leq n$.

\smallskip
Now let $\relcat(\iota_X) = n$.
We can build the same diagrams, except that $\xi_n$ has only a homotopy section $\sigma$ such that $\sigma \circ \iota_X \simeq \iota_n$.
By the Prism lemma \ref{prisme}, $S_n \to S$ has a homotopy section $\tilde\sigma$ such that 
$\tilde\sigma \circ \kappa_S \simeq \tilde\iota_n$. By \propref{catretractcone}, this means $\relcat(\kappa_S) \leq n$.
\end{proof}

Observe that it is {\em not} true that $\secat(\kappa_S) \leq \secat(\iota_X)$. For instance, if $\cM$ is pointed with $\ast$ as zero object,
choose $A$ so that its suspension is not contractible, and choose $X = B = \ast$, hence $S = \Sigma A$; 
then $\secat(\iota_\ast) = 0$ while $\secat(\kappa_{\Sigma A}) = \cat(\Sigma A) = 1$.

\begin{corollary}\label{minorecatcofibre}Assume $\cM$ is pointed, and let $A\to X\to C$ be a homotopy cofibration. Then $\cat(C) \leq \relcat(X,A)$ and $\Cat(C) \leq \Relcat(X,A)$.\end{corollary}

\begin{corollary}
Assume that $\iota_X\colon A \to X$ is a `homotopy retract' of $\kappa_Y\colon B \to Y$, i.e. there exists a homotopy commutative diagram in $\cM$:
$$\xymatrix{
A\ar[r]^t\ar[d]_{\iota_X}&B\ar[d]_{\kappa_Y}\ar[r]^\zeta&A\ar[d]^{\iota_X}\\
X\ar[r]_s&Y\ar[r]_f&X
}$$
such that $f\circ s \simeq \id_X$ and $\zeta\circ t\simeq \id_A$.
Then $\relcat(\iota_X) \leq \relcat(\kappa_Y)$.
\end{corollary}
\begin{proof}
Consider the homotopy pushout $S$ of $\kappa_Y$ and $\zeta$:
$$\xymatrix{
B\ar[d]_{\kappa_Y}\ar[r]^\zeta&A\ar[d]^{\iota_S}\\
Y\ar[r]_q&S
}$$ 
and let $j\colon S \to X$ be the whisker map of $f$ and $\iota_X$.
We have $j \circ q \circ s \simeq f \circ s\simeq \id_X$, so $q\circ s$ is a homotopy section of $j$.
Also we have $q \circ s \circ \iota_X \simeq q \circ \kappa_Y \circ t \simeq \iota_S \circ \zeta \circ t \simeq \iota_S$.
This means that $\iota_X$ is relatively dominated by $\iota_S$,
and we obtain $\relcat(\iota_X) \leq \relcat(\iota_S)$ by \propref{catdominated}.
But we also know that $\relcat(\iota_S) \leq \relcat(\kappa_Y)$ by \lemref{cylsomme}.
Hence $\relcat(\iota_X) \leq \relcat(\kappa_Y)$.
\end{proof}

\begin{lemma}\label{jointmagique}
If $\iota_Y\colon A\to Y$ and $\phi\colon M \to Y$  are maps of $\cM$, consider the following join construction:
$$
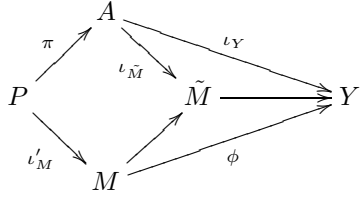
\xymatrix@R=1.5pc@C=1.5pc{
&A\ar[dr]_{\iota_{\tilde M}}\ar[rrrd]^{\iota_Y}\\
P\ar[rd]_{\iota_M'}\ar[ur]^\pi&&\tilde M\ar[rr]^(.35){}&&Y\\
&M\ar[ru]\ar[rrru]_{\phi}
}$$
where the outside square is a homotopy pullback, the inside square is a homotopy pushout,
and the map $\tilde M \to Y$ is the whisker map induced by the homotopy pushout. We have
$$\Relcat(\iota_{\tilde M}) \leq \Relcat(\iota_M') \leq \Relcat(\iota_Y)$$ and
$$\relcat(\iota_{\tilde M}) \leq \relcat(\iota_M') \leq \relcat(\iota_Y).$$
\end{lemma}

\begin{proof}
Let $\Relcat(\iota_X) = n$. Consider the sequence of homotopy pushouts as in
\defref{strongrelcat} and the sequence of maps $\xi_i$ defined after \defref{strongsecat}. This gives the bottom part of the next diagram.
We extend it to a sequence of homotopy commutative diagrams, for $0\leq i< n$,
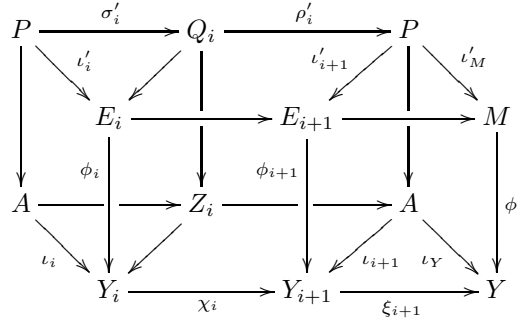
$$\xymatrix@C=1.5pc@R=1.5pc{
P\ar[rr]^{\sigma_i'}\ar[dd]\ar[rd]^(.54){\iota_{i}'}&&Q_{i}\ar[rr]^{\rho_i'}\ar[ld]\ar[dd]|\hole&&P\ar[rd]^{\iota_M'}\ar[ld]_(.54){\iota_{i+1}'}\ar[dd]|\hole&\\
&E_{i}\ar[rr]\ar[dd]_(.3){\phi_{i}}&&E_{i+1}\ar[dd]_(.3){\phi_{i+1}}\ar[rr]&&M\ar[dd]^\phi\\
A\ar[rr]|(.5)\hole\ar[rd]_{\iota_{i}}&&Z_{i}\ar[ld]\ar[rr]|(.5)\hole&&A\ar[rd]_{\iota_Y}\ar[ld]^{\iota_{i+1}}&\\
&Y_{i}\ar[rr]_{\chi_{i}}&&Y_{i+1}\ar[rr]_{\xi_{i+1}}&&Y
}$$
 building all vertical faces as homotopy pullbacks; since $\xi_n = \id_X$ , we may assume $\phi_n = \phi$. 
From the Cube axiom \ref{cube}, we know that the top face of the inside cube is a homotopy pushout.
Notice that $\sigma_i'$ is a homotopy section of $\rho_i'$;
also, since $\iota_0$ has a homotopy inverse, $\iota_0'$ has a homotopy inverse, too;
and finally $\iota_n'\colon P\to E_n$ is $\iota_M'\colon P \to M$. This means $\Relcat(\iota_M') \leq n$.

\smallskip
Now let $\relcat(\iota_Y) = n$.
We can build the same diagrams, except that  $\xi_n$ has only a homotopy section $\sigma$ such that $\sigma \circ \iota_Y \simeq \iota_n$.
 By the Prism lemma \ref{prisme}, the map $E_n \to M$ has a homotopy section $\sigma'$ such that $\sigma' \circ \iota_M' \simeq \iota_n'$. 
By \propref{catretractcone}, this means $\relcat(\iota_M') \leq n$.

\smallskip
The remaining inequalities are direct applications of \lemref{cylsomme}.
\end{proof}

Notice the two interesting particular cases:

\begin{corollary}Assume $\cM$ is pointed. Let  $f\colon F \to E$ be the homotopy fibre of $E\to B$.
 Then $\cat(E/F) \leq \relcat(f) \leq \cat(B)$ and $\Cat(E/F) \leq \Relcat(f) \leq \Cat(B)$.\end{corollary}

\begin{corollary}\label{relcatganea}
Consider the Ganea construction of any map $\iota_X\colon A\to X$ made in \defref{ganea}.
For any $i \geq 0$, we have $\relcat(\alpha_{i+1})\leq \relcat(\beta_i) \leq \relcat(\iota_X)$.
\end{corollary}

We end this subsection by observing that, clearly:

\begin{lemma}\label{cylcone}
Suppose we have a homotopy commutative diagram
where the square is a homotopy pushout:
$$\xymatrix{
A\ar[r]^\sigma\ar[rd]_{\iota_X}&Z\ar[r]^{\rho}\ar[d]^{\tau}&A\ar[d]^{\iota_C}\\
&X\ar[r]_\chi&C 
}$$
and where $\rho\circ\sigma\simeq\id_A$.
Then, we have 
$\Relcat(\iota_C)\leq \Relcat(\iota_X)+1$.
\end{lemma}

\subsection{Comparing all these invariants}

Notice that, for any map $\iota_X\colon A\to X$, \propref{catretractstrong} implies $\relcat(\iota_X)\leq \Pushcat(\iota_X)$. On the other hand, we have obvious inequalities: $\secat(\iota_X) \leq \relcat(\iota_X)$ and $\Pushcat(\iota_X)\leq \Relcat(\iota_X)$. One might think that these four integers could be quite different; indeed, for instance, $\secat(\iota_X) = 0$ iff $\iota_X$ has a homotopy section, while $\relcat(\iota_X) = 0$ iff $\iota_X$ is a homotopy equivalence. But in fact the four integers can  differ only by 1, as is shown by the following result, which is an enhancement of a classical result of F.~Takens \cite{Tak70}.

\begin{theorem}\label{takens}
For any map $\iota_X\colon A\to X$  of $\cM$,
we have:
$$\secat(\iota_X)\leq \relcat(\iota_X) \leq \Pushcat(\iota_X) \leq \Relcat(\iota_X)\leq \secat(\iota_X)+1.$$
\end{theorem}

\begin{proof}
We have just observed the first three inequalities. 

Let $\secat(\iota_X)=n$ and let $\sigma\colon X\to G_n$ be a homotopy section of the Ganea map $g_n\colon G_n\to X$. 
Use $\sigma$ as $\phi$ and $\alpha_n: A \to G_n$ as $\iota_Y$ in \lemref{jointmagique} to get $\tilde X = X \vee_P A$ with $\Relcat(\iota_{\tilde X}) \leq \Relcat(\alpha_n) \leq n$.

By definition of $P$, we have a homotopy commutative diagram
$$\xymatrix{
P\ar[r]^{\iota_X'}\ar[d]_\pi&X\ar[d]_\sigma\ar@{=}[dr]&\\
A\ar[r]^{\alpha_n}\ar@/_1pc/[rr]_{\iota_X}&G_n\ar[r]^-{g_n}&X
}$$
therefore the map
$\iota_X'\colon P\to X$ factors through $\iota_X\colon A\to X$ up to homotopy.
This factorization allows the construction of the following homotopy commutative diagram where each square is a homotopy pushout:
$$\xymatrix{
P\ar[r]^\pi\ar[d]_\pi&A\ar[d]_b\ar@{=}[dr]\\
A\ar[r]_a\ar[d]_{\iota_X}&\tilde A\ar[d]_f\ar[r]_j&A\ar[d]^{\iota_X}\\
X\ar[r]_s&\tilde X\ar[r]_r&X
}$$
We have $\iota_X\circ \pi \simeq \iota_X'$, so $f\circ b \simeq \iota_{\tilde X}$ by the Prism lemma \ref{prisme}.
The map $j$ is the whisker map of two copies of $\id_A$ induced by the homotopy pushout $\tilde A$; so $j\circ b \simeq \id_A$.
As a consequence of  \lemref{cylcone}, we have $\Relcat(\iota_X)\leq \Relcat(\iota_{\tilde X})+1 \leq n + 1$.
\end{proof}

We recover \propref{cylcofibre} as a corollary of \thmref{takens}:

\begin{corollary}
Assume $\cM$ is pointed. Given a homotopy cofibration sequence $A \to B \to C$, we have $\Relcat(C,B) \leq 1$.
\end{corollary}

\begin{proof}
Since the map $\iota_\ast \colon A\to \ast$ to the zero object has a homotopy section $\ast \to A$, we have $\secat(\iota_\ast) = 0$.
So \thmref{takens} gives $\Relcat(\iota_\ast) \leq 1$, and \lemref{cylsomme} gives the result.
\end{proof}

The following corollary shows that the sectional and relative categories of a map differ whenever the category of its homotopy cofibre is greater than the category of its target:

\begin{corollary}\label{cofdiff}
Assume $\cM$ is pointed. For any map $\iota_X\colon A \to X$ with homotopy cofibre $C$ such that $\cat(X) < \cat(C)$, we have $\secat(\iota_X) = \cat(X)$ and $\relcat(\iota_X) = \cat(C) = \cat(X) + 1$. 
\end{corollary}

\begin{proof}
 By \propref{secatbut}, we have $\secat(\iota_X) \leq \cat(X)$ and by  \cororef{minorecatcofibre}, we have $\cat(C) \leq \relcat(\iota_X)$. So, by the hypothesis, $\secat(\iota_X)$ and $\relcat(\iota_X)$ must differ at least by 1. On the other hand, by \thmref{takens}, $\secat(\iota_X)$ and $\relcat(\iota_X)$ can differ at most by 1. Hence we obtain the desired equalities.
\end{proof}

\begin{example}\label{hopf}
The homotopy cofibre of the Hopf fibration $h\colon S^3 \to S^2$ is $\CP^2$ and we have $\cat(S^2) = 1 < \cat(\CP^2) = 2$. 
Thus $\secat(h) = 1$ and $\relcat(h) = 2$.
\end{example}

In particular, \cororef{cofdiff} shows that the category $\cat(C)$ of the homotopy cofibre $C$ of any map $A\to X$ is always less than or equal to $\cat(X)+1$ (a well-known result in the context of topological spaces). We extend this with the following proposition:

\begin{proposition}\label{relcatcone}
With the same notations and hypotheses as in \lemref{pushcone}, we have $$\relcat(\iota_C)\leq \relcat(\iota_X)+1.$$
\end{proposition}

\begin{proof}
Suppose $\relcat(\iota_X)\leq n$ and consider the section $\sigma\colon X\to G_n(\iota_X)$ of the map $g_n\colon G_n(\iota_X)\to X$ given by \defref{LSganea}.
 Construct $\tilde X$ with $\Relcat(\tilde X,A) \leq n$ as in the proof of \thmref{takens}.
 Since $\sigma \circ \iota_X \simeq \alpha_n$ we get a whisker map $A \to P$:
$$\xymatrix{
A\ar@{.>}[r]\ar@{=}[rd]\ar@/^1pc/[rr]^{\iota_X}&P\ar[r]\ar[d]^{\pi}&X\ar[d]^{\sigma}\\
&A\ar[r]_{\alpha_n}&G_n 
}$$
thus $\pi$ is a homotopy epimorphism, which implies that $a \simeq b$;
therefore $\iota_{\tilde X} \simeq f\circ b \simeq f \circ a \simeq s \circ \iota_X$.
Now consider the following homotopy pushouts
$$\xymatrix{
{}&A\ar[r]^a\ar[d]_{\iota_X}&\tilde A\ar[r]^j\ar[d]_f&A\ar[d]_{\iota_X}\ar@/^1.5pc/[dd]^{\iota_C}\\
Z\ar[r]^{\tau}\ar[d]_{\rho}&X\ar[d]_\chi\ar[r]^s&
\tilde X\ar[d]_{\tilde\chi}\ar[r]^r&X\ar[d]_\chi\\
A\ar[r]_{\iota_C}&C\ar[r]_{s'}&\tilde C\ar[r]_-{r'}&C
}$$
Recall that $j\circ a \simeq \id_A$, so $r\circ s \simeq \id_X$ and $r'\circ s' \simeq \id_C$; we deduce that $\iota_C\colon A\to C$ is relatively dominated by $s' \circ \iota_C\colon A\to \tilde C$. 
On the other hand, we have a homotopy pushout
$$\xymatrix{%
Z\ar[r]^{\rho}\ar[d]_{s\,\tau}&A\ar[d]^{s'\,\iota_C}\\  
\tilde X\ar[r]_{\tilde\chi}&\tilde C
}$$
Moreover $\tilde \chi \circ \iota_{\tilde X} \simeq \tilde \chi \circ s \circ \iota_X  \simeq s' \circ \chi \circ \iota_X \simeq s' \circ\iota_C$; 
so by \lemref{pushcone} we see that $\Pushcat(\tilde C,A)\leq n+1$.
 Finally we deduce from \propref{catretractstrong} that 
$\relcat(C,A)\leq n+1$.
\end{proof}

We already know that $\relcat(G_i(\iota_X),A)$ is less than or equal to
both $i$ and $\relcat(X,A)$. In fact, one can make this relation more precise.
The following is an extension of a result of O.~Cornea \cite{Cor94}.

\begin{proposition}\label{cornea}
Let $\iota_X\colon A\to X$ be any map of $\cM$.
Consider the map $\alpha_i\colon A\to G_i(\iota_X)$ of the Ganea construction.
We have:
$$\relcat(\alpha_i)= \min\{ i , \relcat(\iota_X) \}.$$
\end{proposition}

\begin{proof}
Let $\relcat(X,A)=n$.

Suppose $i\geq n$. The map $g_n\colon G_n(\iota_X)\to X$ has a homotopy section $\sigma$ such that $\sigma \circ \iota_X \simeq \alpha_n$, 
therefore $g_i\colon G_i(\iota_X) \to X$ has a homotopy section $\sigma_i = \gamma_{i-1} \circ \dots \circ \gamma_n \circ \sigma$ and $\sigma_i \circ \iota_X \simeq \alpha_i$. So we have $\relcat(X,A)\leq \relcat(G_i(\iota_X),A)$ from \cororef{catdominated}.
On the other hand, \cororef{relcatganea} gives us the reverse inequality, and so the equality is proved.

Now suppose $i<n$, and assume that $\relcat(G_i(\iota_X),A)\leq i-1$. From \propref{relcatcone}, we deduce that $\relcat(G_{i+1}(\iota_X),A)\leq i$, and so on, until we obtain $\relcat(G_n(\iota_X),A) \leq n-1$. This contradicts $\relcat(G_n(\iota_X),A) \geq n$ established above, so we have $\relcat(G_i(\iota_X),A)>i-1$.
On the other hand, clearly, we have $\relcat(G_i(\iota_X),A)\leq i$, and the equality is proved.
\end{proof}

\subsection{The Whitehead construction}

In this subsection, we assume that the category $\cM$ is pointed.

\begin{definition}\label{whitehead}
For any map $\iota_X\colon A \to X\in\cM$, the \emph{Whitehead
construction} or {\em fat wedges} of $\iota_X$ is the
following sequence of homotopy commutative diagrams ($i> 0$):
$$\xymatrix{
&X^i\times A
\ar[rd]_{\upsilon_i}
\ar[rrrd]^{{\rm id}_{X^i} \times \iota_X}&&&\\
T_{i-1}\times A
\ar[ru]^{t_{i-1}\times{\rm id}_A}
\ar[rd]_{{\rm id}_{T_{i-1}}\times\iota_X}
&&T_i\ar[rr]|-(.35){t_i}&&X^{i+1}\\
&T_{i-1}\times X\ar[ru]
\ar[rrru]_{t_{i-1}\times{\rm id}_X}&&&
}$$
where the outside square is a homotopy pullback, the inside square is a homotopy pushout,
and the map $t_i\colon T_i \to X^\ipu$ is the whisker map induced by this homotopy pushout.
The induction starts with $t_0 = \iota_X\colon A \to X$.
\end{definition}

We  denote $T_i$ by $T_i(\iota_X)$, or by $T_i(X,A)$. We also write $T_i(X) = T_i(X,\ast)$.

\smallskip
The following result is (almost) Theorem 8 in \cite{GarVan10} :

\begin{theorem}\label{ganeawhitehead}
Let $\iota_X\colon A\to X$ be a map of $\cM$. Let $i \geq 0$ and consider the
diagonal map $\Delta_\ipu\colon X \to X^\ipu$. Then we have homotopy pullbacks:
$$\xymatrix{A\ar[r]^(0.45){\alpha_i}\ar[d]_{\delta_i}&G_i(\iota_X)\ar[r]^(0.55){g_i}\ar[d]^{\varepsilon_i}&X\ar[d]^{\Delta_\ipu}\\
X^i \times A\ar[r]_{\upsilon_i}&T_i(\iota_X)\ar[r]_{t_i}&X^{i+1}
}$$ 
\end{theorem}

\begin{proof}
We proceed inductively. Consider the following homotopy commutative diagram, where $\delta_i$ is the whisker map of $\Delta_i\circ \iota_X$ and $\id_A$, and $\bar \epsilon_{i-1}$ is the whisker map of $\epsilon_{i-1}$ and $g_{i-1}$:
$$\xymatrix{
&A\ar[r]^{\iota_X}\ar[d]_{\delta_i}&X\ar[d]^{\Delta_\ipu}&G_{i-1}\ar[l]_{g_{i-1}}\ar[d]^{\bar\epsilon_{i-1}}\\
&X^i \times A\ar[r]\ar[ld]&X^i \times X\ar[ld]\ar[rd]&T_{i-1}\times X\ar[l]\ar[rd]&\\
A\ar[r]^(.6){\iota_X}&X&&X^i&\ar[l]_(.6){t_{i-1}}T_{i-1}
}$$
Applying the Prism lemma \ref{prisme} to the left and right parts of the diagram to obtain that the two upper squares  are homotopy pullbacks.
Now apply the Join theorem \ref{joint} to the two upper squares to get the inductive step.
\end{proof}

We denote $\epsilon_i$ by $\epsilon_i(\iota_X)$ or $\epsilon_i(X,A)$; we also write $\epsilon_i(X) = \epsilon_i(X,\ast)$. Notice that  $\delta_i(\id_X) \simeq \epsilon_i(\id_X) \simeq \Delta_{i+1}$.

We will also denote $\tau_i \simeq \epsilon_i \circ \alpha_i \simeq \upsilon_i \circ \delta_i$. Notice that $\alpha_i\colon A \to G_i$ is nothing but the whisker map of $\tau_i\colon A \to T_i$ and $\iota_X \colon A \to X$ induced by the right homotopy pullback of \thmref{ganeawhitehead}.

\thmref{ganeawhitehead} allows to give a `Whitehead version' of sectional and relative categories:

\begin{proposition}
Let $\iota_X\colon A \to X$ be a map of $\cM$. 

1) We have $\secat(\iota_X) \leq n$ if and only if there exists a map $\rho: X \to T_n(\iota_X)$ such that $t_n \circ \rho \simeq \Delta_{n+1}$.

2) We have $\relcat(\iota_X) \leq n$ if and only if there exists a map $\rho: X \to T_n(\iota_X)$ such that $t_n \circ \rho \simeq \Delta_{n+1}$ and $\rho \circ \iota_X \simeq \tau_n$.
\end{proposition}

\begin{proof}
If $\secat(\iota_X) \leq n$, we have a map $\sigma: X \to G_n$ such that $g_n \circ \sigma \simeq \id_X$. Set $\rho \simeq \epsilon_n \circ \sigma$. We have $t_n \circ \rho \simeq  t_n \circ \epsilon_n \circ \sigma \simeq \Delta_{n+1} \circ g_n \circ \sigma \simeq  \Delta_{n+1}$.
Moreover if $\relcat(\iota_X) \leq n$, we have also $\sigma \circ \iota_X \simeq \alpha_n$, thus $\rho \circ \iota_X \simeq \epsilon_n \circ \sigma \circ \iota_X \simeq \epsilon_n \circ \alpha_n \simeq \tau_n$.

For the reverse direction, assume we have a map $\rho: X \to T_n$ such that $t_n \circ \rho \simeq \Delta_{n+1}$.The right homotopy pullback of \thmref{ganeawhitehead} induces a whisker map $\sigma: X \to G_n$ of $\rho$ and $\id_X$, so $g_n \circ \sigma \simeq \id_X$ and $\epsilon_n \circ \sigma \simeq \rho$. 
If moreover we have $\rho \circ \iota_X \simeq \tau_n$, 
then extend the left homotopy pullback of \thmref{ganeawhitehead} to the following homotopy commutative diagram:
$$\xymatrix{
A\ar@{.>}[r]_{\hat\sigma}\ar[rdd]_{\delta_n}\ar@/^1pc/[rr]^{\iota_X}&P\ar[r]\ar[d]^{\pi}&X\ar[d]_{\sigma}\ar@/^1pc/[dd]^{\rho}\\
&A\ar[r]_{\alpha_n}\ar[d]^{\delta_n}&G_n\ar[d]_{\epsilon_n}\\
&X^n\times A\ar[r]_(.6){\upsilon_n}&T_n
}$$
where both squares (and the rectangle) are homotopy pullbacks and, as the outer diagram commutes up to homotopy, we have a whisker map $\hat\sigma\colon A\to P$ of $\delta_n$ and $\iota_X$; so we have $\delta_n\simeq \delta_n\circ \pi \circ \hat\sigma$. As $\delta_n$ has an obvious (homotopy) retraction, it is a homotopy monomorphism, thus $\id_A \simeq \pi \circ \hat\sigma$. Finally $\sigma\circ\iota_X \simeq \alpha_n \circ\pi \circ \hat\sigma \simeq \alpha_n$.
\end{proof}

\subsection{Change of base}
We return to sectional category and look at its behaviour when the base of the map changes. 
The precise statement, \propref{inegalitespaires}, is a consequence of the following result:

\begin{lemma}\label{changebase}
Suppose we are given any homotopy commutative diagram in $\cM$:
$$\xymatrix@C=3pc{
B\ar[r]^{\kappa_Y}\ar[d]_\zeta&Y\ar[d]^f\\
A\ar[r]_{\iota_X}&X
}$$
For any $i \geq 0$, there is a homotopy commutative diagram in $\cM$:
$$\xymatrix@C=4pc{
G_i(\kappa_Y)\ar[r]^{g_i(\kappa_Y)}\ar[d]_{\zeta_i}&Y\ar[d]^f\\
G_i(\iota_X)\ar[r]_{g_i(\iota_X)}&X
}$$
Moreover, if the first diagram is a homotopy pullback, the second one is a homotopy pullback as well.
\end{lemma}

\begin{proof}
Use the Join theorem \ref{joint} inductively.
\end{proof}

The three following propositions are straightforward consequences of the previous lemma.

\begin{proposition}\label{inegalitespaires}
Suppose we are given any homotopy commutative diagram in $\cM$:
$$\xymatrix@C=3pc{
B\ar[r]^{\kappa_Y}\ar[d]_\zeta&Y\ar[d]^f\\
A\ar[r]_{\iota_X}&X
}$$

1) If $f$ has a homotopy section, then $\secat(\iota_X)\leq\secat(\kappa_Y)$.

2) If the square is a homotopy pullback, then $\secat(\kappa_Y)\leq\secat(\iota_X)$.

3) If $f$ and $\zeta$ have homotopy inverses, then $\secat(\iota_X) = \secat(\kappa_Y)$.
\end{proposition}

\begin{proposition}\label{secatbut}
If a map $\kappa_X\colon B \to X$ factors through $\iota_X\colon A \to X$ up to homotopy, then $\secat(\iota_X) \leq \secat(\kappa_X)$.
\end{proposition}

In particular, if $\cM$ is pointed, for any map $\iota_X \colon A \to X$, $\secat(X,A) \leq \cat(X)$.

It is of course {\em not} true that $\relcat(X,A) \leq \cat(X)$; consider $A
\to \ast$, or see \examref{hopf} for instance.

\begin{proposition}\label{svarc}
Let $\iota_X\colon A\to X$ and $f\colon Y\to X$ be maps of $\cM$ and consider the homotopy pullback $F = Y \times_X A$.
We have $\secat(Y,F)\leq n$ if and only if $f$ factors through  $g_n\colon G_n(\iota_X)\to X$ up to homotopy.
\end{proposition}

This last proposition is an extension of a result of A.S.~Schwarz \cite{Sva66}.

\begin{definition}Assume $\cM$ is pointed.  Let $f\colon Y \to X$ be a map of $\cM$. The \emph{category} of $f$ is the least integer $n$ such that
$f$ factors through $g_n\colon G_n(X) \to X$ up to homotopy.\end{definition}

The category of $f$ is denoted by $\cat(f)$. By \propref{svarc}, $\cat(f) = \secat(Y,F)$ where $F\to Y$ is the homotopy fibre of $f$.

\subsection{Sectional and relative categories of maps in the Ganea construction}
Consider the Ganea construction of $\iota_X\colon A\to X$. 
The following results determine the sectional category and the relative category of the maps $\alpha_i\colon A\to G_i$ and $\beta_i\colon F_i \to G_i$.

\begin{theorem}\label{secatganea}  Let $\iota_X\colon A \to X$ be any map of $\cM$.
 Consider the map $\beta_i\colon F_i(\iota_X)\to G_i(\iota_X)$ of the Ganea construction. We have:
$$\relcat(\beta_i) = \min\{i+1,\relcat(\iota_X)\} \text{ and } \secat(\beta_i) = \min\{i,\secat(\iota_X)\}.$$
\end{theorem}

\begin{proof}
Let $i \geq \secat(\iota_X)$.
By \thmref{takens}, $i+1 \geq \relcat(\iota_X)$. Then by \propref{cornea} and \cororef{relcatganea}, we have $\relcat(\iota_X) = \relcat(\alpha_{i+1}) \leq \relcat(\beta_i) \leq \relcat(\iota_X)$. On the other hand, since $F_i = G_i \times_X A$ and $g_i$ has a homotopy section, \propref{inegalitespaires} yields $\secat(\beta_i) = \secat(\iota_X)$.

Let $i < \secat(\iota_X) \leq \relcat(\iota_X)$.
Then by \propref{cornea} and \cororef{relcatganea}, we have $i+1 = \relcat(\alpha_{i+1}) \leq \relcat(\beta_i)$. On the other hand, since $\alpha_i \simeq \beta_i \circ \theta_i$, by \propref{secatbut} and \propref{cornea}, we have $\secat(\beta_i) \leq \secat(\alpha_i) \leq \relcat(\alpha_i) = i$. By \thmref{takens},  $\secat(\beta_i)$ and $\relcat(\beta_i)$ can only differ by 1, so $\secat(\beta_i) =i$ and $\relcat(\beta_i) = i+1$.
\end{proof}

As a particular case, we have:

\begin{corollary}\label{secatfibreganea}Assume  $\cM$ is pointed. Consider the homotopy fibre $\beta_i\colon F_i(X) \to G_i(X)$ of the map $g_i\colon G_i(X) \to X$. We have
$$\relcat(\beta_i) = \min\{i+1,\cat(X)\} \text{ and } \cat(g_i) = \secat(\beta_i) = \min\{i,\cat(X)\}.$$
\end{corollary}

We are now ready to enhance \propref{cornea}.

\begin{corollary}   \label{lastbutnotleast}
Let $\iota_X\colon A\to X$ be any map of $\cM$.
Consider the map $\alpha_i\colon A\to G_i(\iota_X)$ of the Ganea construction.
We have:
$$\min\{ i , \secat(\iota_X) \} \leq \secat(\alpha_i) \leq  \relcat(\alpha_i)= \min\{ i , \relcat(\iota_X) \}$$
\end{corollary}

\begin{proof}
The last equality is \propref{cornea}.
On the other hand, by \propref{secatbut} and  \thmref{secatganea}, we have $\secat(\alpha_i) \geq \secat(\beta_i) = \min\{ i , \secat(\iota_X)\}$.
\end{proof}

One might expect that the first inequality would be an equality, but it is not.
When $i > \secat(\iota_X)$ the inequality can be strict. 
Indeed consider $\iota_\ast\colon A\to \ast$. We have $\alpha_1\colon A \to A\bowtie A$ (join of $A$ with itself), which is a null map, i.e. it factors through the zero object,
so it can not have a homotopy section (unless $A \simeq \ast$).
Thus we have $\min\{ 1 , \secat(\iota_\ast) \} = \secat (\iota_\ast) = 0$, while $\secat(\alpha_1) = 1$.


\section[Complexity]{Complexity}\label{sectioncomplexity}

In this section, $\cM$ is pointed with $\ast$ as zero object.

\subsection{Complexity}

\begin{definition}\label{compl}Let $X$ be any object of $\cM$.

We define the  \emph{complexity} of $X$ to be the sectional category of the diagonal map $\Delta\colon X \to X \times X$.

Analogously, we define the  \emph{relative} (\resp\emph{pushout}, and \emph {strong}) \emph{complexity} of $X$ to be the relative (\resp pushout, and strong relative) category of the diagonal.
\end{definition}

We use the following notations: $\compl(X) = \secat(\Delta)$, $\relcompl(X) = \relcat(\Delta)$, $\Pushcompl(X) = \Pushcat(\Delta)$, $\Compl(X) = \Relcat(\Delta)$.

\smallskip
In $\Top$, the complexity is called \emph{topological complexity} by M.~Farber\cite{Far03} (up to a shift by 1): ${\rm TC}(X) = \compl(X)+1$.

\smallskip
Consider the diagonal $\Delta_{i+1} \colon X\to X^{i+1}$ and the maps  $\delta_i(\iota_X)\colon A \to X^i \times A$ and $\epsilon_i(\iota_X)\colon G_i(\iota_X) \to T_i(\iota_X)$ built for any map  $\iota_X\colon A \to X$ in \thmref{ganeawhitehead}. 

\begin{proposition}\label{catsecat}
For any object $X$ and any map $\iota_X\colon A \to X$  of $\cM$, 
$$\cat(X^i) \leq \secat(\delta_i(\iota_X))  \leq \secat(\epsilon_i(\iota_X)) \leq \secat(\Delta_\ipu) \leq \cat(X^\ipu).$$
\end{proposition}

\begin{proof}
The last inequality is just \propref{secatbut}.

On the other hand, we have the following homotopy pullbacks:
$$\xymatrix{%
\ast\ar[r]\ar[d]&A\ar[r]\ar[d]^{\delta_i}&G_i(\iota_X)\ar[r]^{g_i}\ar[d]^{\varepsilon_i}&X\ar[d]^{\Delta_\ipu}\\
X^i \ar[r]_(.45){\inc_1}&X^i \times A\ar[r]&T_i(\iota_X)\ar[r]_{t_i}&X^\ipu
}$$
The two right squares are given by \thmref{ganeawhitehead}, and the left square is easily obtained by the Prism lemma \ref{prisme}.
We deduce the first three inequalities from \propref{inegalitespaires}.
\end{proof}

This  proposition suggests the following definition of complexity of a map:

\begin{definition}For any map $\iota_X\colon A \to X$, we  define the {\em complexity} of $\iota_X$ as the sectional category of $\delta_1(\iota_X)\colon A \to X \times A$.\end{definition}
Recall that $\delta_1(\iota_X)$ is the whisker map of $\iota_X$ and $\id_A$. 
We write $\compl(\iota_X) = \compl(X,A) = \secat(\delta_1(\iota_X))$.

In particular $\compl(\id_X) = \compl(X)$, since $\delta_1(\id_X) \simeq \Delta$.
On the other hand $\compl(X,*) = \cat(X)$; indeed consider $\epsilon_1(X)\colon \Sigma\Omega X \to X\vee X$, 
and we see that $\cat(X) \leq \secat(\delta_1(X)) \leq \secat(\epsilon_1(X)) \leq \cat(X\vee X) = \cat(X)$. 

\propref{catsecat} gives:

\begin{corollary}\label{ineqcompl}For any object $X$ and any map $\iota_X\colon A\to X$ of $\cM$,
$$\cat(X)\leq \compl(\iota_X) \leq \compl(X) \leq \cat(X\times X).$$
\end{corollary}

\begin{example}M.~Farber \cite{Far03} has shown that the complexity of a sphere is 1 if the dimension is odd and 2 if the dimension is even.

Consider the Hopf fibration $S^7\to S^4$ and factor by the action of $S^1$ on $S^7$ to get $\iota\colon \CP^3\to S^4$. Let $u$ be a generator of the rational cohomology $H^\ast(S^4)$ and $v$ be a generator of $H^\ast(\CP^3)$. Define $a = 1 \otimes v^2 - u \otimes 1$ and $b = u \otimes v^2$ in $H^\ast(S^4) \otimes H^\ast(\CP^3)$. The map $\delta_1\colon \CP^3 \to S^4 \times \CP^3$ induces $\delta_1^\ast \colon H^\ast(S^4) \otimes H^\ast(\CP^3) \to  H^\ast(\CP^3)$. We have $\delta_1^\ast(a) = 1.v^2 - v^2.1 = 0$ and $a^2 = -2 b \not= 0$. By A.S.~Schwarz \cite{Sva66}, Theorem 4, this means that $2 \leq \secat(\delta_1) = \compl(\iota)$. But by \cororef{ineqcompl}, we also know that  $\compl(\iota) \leq \cat(S^4 \times S^4) = 2$. So $\compl(\iota) = 2$.

In contrast, from the fact that $\compl(S^{2n+1}) = 1$ and \cororef{ineqcompl}, we see that $\compl(\iota') = 1$ for {\em any} map $\iota' \colon A \to S^{2n+1}$.
\end{example}

On the other hand, by \thmref{takens}, we know that all the variants of complexity can only differ by 1:

\begin{proposition}\label{takenscompl}For any object $X$ of $\cM$, 
$$\compl(X) \leq \relcompl(X) \leq \Pushcompl(X) \leq \Compl(X) \leq \compl(X)+1.$$
\end{proposition}

\smallskip Observe the following other lower bound of the strong complexity:

\begin{proposition}\label{CatRelcat}
For any object $X$ and any map $\iota_X\colon A \to X$ of $\cM$, 
$$\Cat(X^{i}) \leq \Relcat(\delta_i(\iota_X)) \leq \Relcat(\epsilon_i(\iota_X)) \leq \Relcat(\Delta_\ipu).$$
\end{proposition}

In particular, $\Cat(X) \leq \Compl(X)$.

\begin{proof}
Use \lemref{jointmagique} with the same homotopy pullbacks as in \propref{catsecat}.
\end{proof}




\subsection{Complexity of a suspension}\label{CompSX}

Let us consider the {\em pinch} map $p\colon \SX \to \SX \vee \SX$ which is the whisker map induced by
the top homotopy pushout in the following homotopy commutative diagram:
$$\xymatrix@R=1pc@C=0.8pc{
&X\ar[ld]\ar[rr]\ar[dd]|!{[ld];[rd]}\hole&&\ast\ar[ld]\ar[rd]\ar[dd]|\hole&\\
\ast\ar[rr]\ar[dd]&&\SX\ar@{..>}[dd]_(.3)p\ar@{=}[rr]&&\SX\ar[dd]^\Delta\\
&\ast\ar[rr]|(.5)\hole\ar[ld]&&\SX\ar[ld]^(.3){\inc_2}\ar[rd]&\\
\SX\ar[rr]_(.4){\inc_1}&&\SX\vee \SX\ar[rr]_{t_1}&&\SX\times \SX
}$$
where $\inc_1$ and $\inc_2$ are the obvious maps.
The outside part of the diagram is described in \examref{suspcoh} in the Appendix, and it is extended to the whole homotopy commutative diagram using \lemref{whiskercube}.

\begin{lemma}\label{complpinch}
We have $\Relcat (p) \leq 1$.
\end{lemma}

\begin{proof}
Observe that the faces of the inside cube in the preceeding diagram are homotopy pushouts (use the Prism lemma \ref{prisme} to show that the front and right squares are indeed homotopy pushouts).
So $p$ appears as the homotopy cofibre of $X \to \SX$ and \propref{cylcofibre} gives the result.
\end{proof}

Actually we have the following explicit homotopy commutative diagram with a homotopy pushout, where $\inc_1$ and $\pro_1$ are the obvious maps:
$$\xymatrix{
\SX\ar[r]^(.4){\inc_1}\ar@{=}[rd]&\SX \vee X\ar[r]^(.55){\pro_1}\ar[d]^{\pro_1}&\SX\ar[d]^{p}\\
&\SX\ar[r]&\SX\vee \SX
}$$
and  \defref{strongrelcat} yields $\Relcat(p) \leq 1$ directly.

\begin{theorem}\label{complsusp}Let $X$ be any object of $\cM$. We have $$\Compl(\SX) \leq 2.$$
\end{theorem}

\begin{proof}
From \propref{cofibrationproduit} we have a homotopy cofibration sequence:
$$\xymatrix{X \bowtie X\ar[r]_(.45)u& \SX\vee \SX \ar[r]_{t_1}& \SX \times \SX}$$
So \propref{doublecofibre} gives the following homotopy pushout:
$$\xymatrix@C=3pc{%
(\SX\vee \SX) \vee (X\bowtie X)\ar[r]^(.6){(\id,u)}\ar[d]_{\pro_1}&\SX\vee \SX\ar[d]^{t_1}\\
\SX\vee \SX\ar[r]_{t_1}&\SX\times \SX
}$$
We now extend this square to the following homotopy commutative diagram:
$$\xymatrix@C=3pc{%
\SX \vee (X\bowtie X)\ar[r]^(.45){p \vee \id}\ar[d]_{\pro_1}&(\SX\vee \SX) \vee (X\bowtie X)\ar[r]^(.6){(\id,u)}\ar[d]_{\pro_1}&\SX\vee \SX\ar[d]^{t_1}\\
\SX \ar[r]_p&\SX\vee \SX\ar[r]_{t_1}&\SX\times \SX
}$$
where $p \vee \id$ is the whisker map of $\inc_1 \circ p$ and $\inc_2$.
The two squares are homotopy pushouts, so the outside rectangle is a homotopy pushout, too.
Recall that $t_1\circ p \simeq \Delta$. We now obtain a homotopy commutative diagram with a homotopy pushout:
$$\xymatrix@C=3pc{
\SX\ar[r]^(.35){\inc_1}\ar[rd]_{p}&\SX \vee (X\bowtie X)\ar[r]^(.6){\pro_1}\ar[d]^{(p,u)}&\SX\ar[d]^{\Delta}\\
&\SX\vee \SX\ar[r]_{t_1}&\SX\times \SX
}$$
 Finally \lemref{complpinch} and \lemref{cylcone} yield the desired result. 
\end{proof}

\begin{example}
As mentioned before, for any $n > 0$, $\compl(S^{2n}) = 2$.
Therefore, by \propref{takenscompl} and \thmref{complsusp}, we see that $\Compl(S^{2n}) = 2$. 
\end{example}


\appendix
\section[Appendix]{Toolbox}

All constructions made in this paper can be achieved in a closed model category $\cM$ satisfying the following additional axiom:

\begin{axiom}[Cube axiom]\label{cube}
For any homotopy commutative diagram in $\cM$:
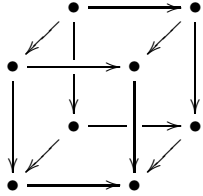
$$\xymatrix@R=1pc@C=1pc{
&\bullet\ar[rr]\ar[ld]\ar[dd]|\hole
&&\bullet\ar[ld]\ar[dd]\\
\bullet\ar[rr]\ar[dd]&&\bullet\ar[dd]&\\
&\bullet\ar[dl]\ar[rr]|!{[ur];[dr]}\hole&&\bullet\ar[dl]\\
\bullet\ar[rr]&&\bullet&
}$$
if the bottom face is a homotopy pushout and the four vertical faces are homotopy pullbacks,
then the top face is a homotopy pushout.\end{axiom}

Of course this axiom is satisfied in the category of topological spaces; see \cite{Mat76}, Theorem 25.

\smallskip
The following lemma is used very often, sometimes implicitly:

\begin{lemma}[Prism lemma]\label{prisme}
Suppose given any homotopy commutative diagram in $\cM$:
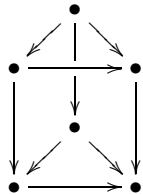
$$\xymatrix@R=1pc@C=1pc{%
&\bullet\ar[rd]\ar[ld]\ar[dd]|\hole&\\
\bullet\ar[dd]\ar[rr]&&\bullet\ar[dd]\\
&\bullet\ar[rd]\ar[ld]&\\
\bullet\ar[rr]&&\bullet
}$$

1) Suppose the left square is a homotopy pushout. Then the front square is a homotopy pushout if and only if the right square is a homotopy pushout.

2) Suppose the front square is a homotopy pullback. Then the left square is a homotopy pullback if and only if the right square is a homotopy pullback.
\end{lemma}

By a `homotopy commutative diagram' we mean not only a set of maps, but also a choice of homotopies between (composites of) maps with same source and target, and which are `compatible' with each other; see \cite{Mat76} or \cite{DoeHa06} for details.

In particular, a `homotopy commutative square' involves the choice of one homotopy.

\begin{example}\label{suspcoh}Consider the following diagram:
$$
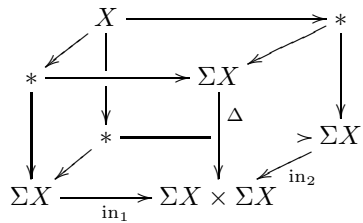
\xymatrix@R=0.8pc@C=0.8pc{
&X\ar[rr]\ar[ld]\ar[dd]|\hole&&\ast\ar[ld]\ar[dd]\\
\ast\ar[rr]\ar[dd]&&{\SX}\ar[dd]^(.3)\Delta&\\
&\ast\ar[dl]\ar[rr]|!{[ur];[dr]}\hole&&{\SX}\ar[dl]^(.4){\inc_2}\\
{\SX}\ar[rr]_(.45){\inc_1}&&{\SX \times \SX}&
}$$
in the category of pointed topological spaces, where $\inc_1$ and $\inc_2$ are the obvious maps and $\Sigma X$ is the reduced suspension of $X$,
with the homotopy $H\colon X\times I \to \SX\colon (x,t) \mapsto [x,t]$ attached to the top, left and back squares, and (of course) the static homotopy attached to the front, right and bottom squares.
The left, back and bottom squares yield a homotopy:
$$K \colon X \times I \to \SX \times \SX \colon (x,t) \mapsto 
\left\{\begin{array}{ll}([x,2t],\ast) &\mbox { if } t \leq \undemi\\[2mm]
(\ast,[x,2t-1]) &\mbox { if } t \geq \undemi\end{array}\right.$$
The top, front and right squares yield a homotopy:
$$L \colon X \times I \to \SX \times \SX \colon (x,t) \mapsto 
([x,t],[x,t])$$
The homotopies are compatible (we write $K \equi L$) thanks to the following higher homotopy:
$$M \colon X\times I \times I \to \SX \times \SX \colon (x,t,s) \mapsto 
\left\{\begin{array}{ll}([x,\frac{2t}{s+1}],\ast) &\mbox { if } t \leq \frac{1-s}{2}\\[2mm]
([x,\frac{2t}{s+1}],[x,\frac{2t+s-1}{s+1}]) &\mbox{ if } \frac{1-s}{2} \leq t \leq \frac{1+s}{2}\\[2mm]
(\ast,[x,\frac{2t+s-1}{s+1}]) &\mbox { if } t \geq \frac{1+s}{2}\end{array}\right.$$
\end{example}

We don't write the homotopies explicitly in the paper because in most cases, all we have to know is that they are there!
Many diagrams are built using homotopy pushouts and/or homotopy pullbacks constructions, and in this case the homotopies are well defined (up to equivalences) and the diagrams are naturally homotopy commutative. \lemref{whiskercube} below illustrates this fact.
However, it is important to keep in mind that all these homotopies are still there and are well defined (up to equivalences).

\smallbreak
We write $f \simeq g$ when the map $f$ is homotopic to $g$.

We denote by $U\vee_V W$  the homotopy pushout of maps $\xymatrix{U&\ar[r]\ar[l]V&W}$ 
and by $A \times_B C$ the homotopy pullback of maps $\xymatrix{A\ar[r]&B&\ar[l]C}$. 

When the category $\cM$ is pointed, i.e. it has a {\em zero} object $\ast$ both initial and terminal, 
we omit to write the subscript on $\vee$ or $\times$  if it is the zero object.
The homotopy pushout $U \vee_V \ast$ is denoted by $U/V$,
the map $U \to U/V$ (or $U/V$ itself) is called the {\em homotopy cofibre} of $V \to U$ 
and the sequence $\xymatrix{V\ar[r]&U\ar[r]&U/V}$ is called a {\em homotopy cofibration}.
Let $F = E \times_B \ast$; the map $F \to E$ (or $F$ itself) is called {\em homotopy fibre} of $E\to B$.
Finally, we write $\Sigma X = *\vee_X *$, called {\em suspension} of $X$, and $\Omega X = * \times_X *$.

\smallbreak
Consider two homotopy commutative squares in $\cM$, where the inside square is a homotopy pushout:
$$\xymatrix@R=1.5pc@C=1.5pc{
&A\ar[dr]_a\ar[rrrd]^f\\
P\ar[rd]_v\ar[ur]^u&&J\ar@{-->}[rr]|-(.35)j&&B\\
&C\ar[ru]^c\ar[rrru]_g
}$$
These two squares {\em and} the two attached homotopies $H\colon a\circ u\simeq c\circ v$ and 
$K\colon f\circ u\simeq g\circ v$ induce a map $j\colon J\to B$, called the \emph{whisker} map, together with homotopies of the triangles $M\colon f\simeq j\circ a$ and $N\colon j\circ c\simeq g$ making the whole diagram homotopy commutative, that is $M\circ (u\times \id_I) + j\circ H + N\circ (v\times \id_I) \equi K$. 
The map $j$ is `universal' in the following sense: If there is another map $j'$ and (other) homotopies $M'\colon f\simeq j'\circ a$ and $N'\colon j'\circ c\simeq g$ of the triangles making the whole diagram homotopy commutative, then there is a homotopy $L\colon j \simeq j'$, and the diagram with all maps, including $j$ and $j'$, and all homotopies, is homotopy commutative, that is
$M+ L\circ (a\times \id_I) \equi M'$ and $L\circ(c\times \id_I) + N' \equi N$.
See \cite{Mat76}, Theorem 11.

The whisker map $j$ is sometimes denoted by $(f,g)$.

Despite this notation, we emphasize that $j$ is determined not only by the outside square but also by the attached homotopy. A different choice of homotopy lead to a different induced map. In other words, the whisker map is unique (up to homotopy) {\em once} the homotopy is fixed. For instance, in the category of pointed topological spaces, if $A$ and $C$ are the one point set $\{\ast\}$, then $J$ is the reduced suspension $\Sigma P$ and the homotopy attached to the homotopy pushout is $H: P\times I \to \Sigma P: (x,t) \mapsto [x,t]$. Let $B=J=\Sigma P$. If the homotopy attached to the outside square is $H$, i.e. the same as the homotopy attached to the inside square, then $j$ is the identity. But if the homotopy attached to the outside square is the static one, then $j$ is the null map.

There is a `dual' notion of whisker map for homotopy pullbacks. 

\begin{definition}[Join]
If in the above diagram the outside square is a homotopy pullback, then
the whisker map $j$ induced by the homotopy pushout (or $J$ itself) is called the {\em join} of $f$ and $g$.
\end{definition}

We denote the join $J$ by $A\bowtie_B C$. We omit the subscript on $\bowtie$ if it is the zero object.

\smallbreak
Ganea and Whitehead constructions are particular cases of join constructions.

\begin{lemma}[Whisker maps inside a cube]\label{whiskercube}Given a homotopy commutative cube:
$$\xymatrix@R=1pc@C=1pc{
&P\ar[rr]\ar[ld]\ar[dd]|\hole&&A\ar[ld]\ar[dd]\\
C\ar[rr]\ar[dd]&&B\ar[dd]^(.3)\phi&\\
&P'\ar[dl]\ar[rr]|!{[ur];[dr]}\hole&&A'\ar[dl]\\
C'\ar[rr]&&B'&
}$$
it can be extended to a homotopy commutative diagram:
$$\xymatrix@R=1pc@C=1pc{
&P\ar[ld]\ar[rr]\ar[dd]|!{[ld];[rd]}\hole&&A\ar[ld]\ar[rd]\ar[dd]|\hole&\\
C\ar[rr]\ar[dd]&&A \vee_P C\ar@{.>}[dd]_(.3){k}\ar@{.>}[rr]_(.35)j&&B\ar[dd]^\phi\\
&P'\ar[rr]|(.5)\hole\ar[ld]&&A'\ar[ld]\ar[rd]&\\
C'\ar[rr]&&A' \vee_{P'} C'\ar@{.>}[rr]_l&&B'
}$$
where $j$ and $k$ are the whisker maps induced by the top homotopy pushout and $l$ is the whisker map induced by the bottom homotopy pushout.
\end{lemma}

\begin{proof}The top part, the inside cube, and the bottom part of the diagram are homotopy commutative by construction of the whisker maps $j$, $k$ and $l$ respectively. Moreover both $l\circ k$ and $\phi \circ j$ make the following diagram homotopy commutative:
$$\xymatrix@R=1pc@C=1pc{
&A\ar[dr]\ar[rr]&&A'\ar[rd]\\
P\ar[rd]\ar[ur]&&A\vee_P C\ar@{-->}[rr]&&B'\\
&C\ar[ru]\ar[rr]&&C'\ar[ru]
}$$
So by the universal property of the whisker map, $l\circ k \simeq \phi \circ j$ and the whole diagram is homotopy commutative. 
\end{proof}

\begin{proposition}\label{doublecofibre}Let $\cM$ be pointed. Let be given any homotopy cofibration sequence $\xymatrix{Y\ar[r]|f&A\ar[r]|g&X}$. There is a homotopy pushout:
$$\xymatrix@C=2.5pc{
A\vee Y\ar[r]^{(\id_A,f)}\ar[d]_{\pro_1}&A\ar[d]^g\\
A\ar[r]_g&X
}$$
where $(\id_A,f)$ is the whisker map of $\id_A$ and $f$.
\end{proposition}

\begin{proof}Apply \lemref{whiskercube} to the outside part of the following diagram:
$$\xymatrix@R=1pc@C=1pc{
&\ast\ar[ld]\ar[rr]\ar[dd]|!{[ld];[rd]}\hole&&Y\ar[ld]_(.6){\inc_2}\ar[rd]^f\ar[dd]|\hole&\\
A\ar[rr]_(.25){\inc_1}\ar@{=}[dd]&&A \vee Y\ar[dd]_(.3){\pro_1}\ar[rr]_(.32){(\id,f)}&&A\ar[dd]^g\\
&\ast\ar[rr]|(.5)\hole\ar[ld]&&\ast\ar[ld]\ar[rd]&\\
A\ar@{=}[rr]&&A\ar[rr]_{g}&&X
}$$
We get the whole diagram homotopy commutative. Moreover, using the Prism lemma \ref{prisme} we get that all  squares, except the back one and the leftern front one, are homotopy pushouts.
\end{proof}

\begin{theorem}[Join theorem]\label{joint}
Suppose we have two homotopy commutative squares:
$$\xymatrix{
A\ar[r]\ar[d]&B\ar[d]&C\ar[l]\ar[d]\\
A'\ar[r]&B'&C'\ar[l]
}$$
Then there is a homotopy commutative diagram:
$$\xymatrix{
A\ar[r]\ar[d]&A\bowtie_{B}C\ar[r]\ar[d]&B\ar[d]\\
A'\ar[r]&A'\bowtie_{B'}C'\ar[r]&B'
}$$
Moreover, if the two squares in the first diagram are homotopy pullbacks, then the two squares in the second diagram are homotopy pullbacks as well.
\end{theorem}

\begin{proof}
We construct the following homotopy commutative diagram
$$\xymatrix@R=0.8pc@C=0.8pc{
&A\times_B C\ar[ld]\ar[rr]\ar@{-->}[dd]|!{[ld];[rd]}\hole&&A\ar[ld]\ar[rd]\ar[dd]|\hole&\\
C\ar[rr]\ar[dd]&&A\bowtie_{B}C\ar@{..>}[dd]\ar[rr]&&B\ar[dd]\\    
&A'\times_{B'} C'\ar[rr]|(.57)\hole\ar[ld]&&A'\ar[ld]\ar[rd]&\\
C'\ar[rr]&&A'\bowtie_{B'}C'\ar[rr]&&B'
}$$
where $\xymatrix@1{A\times_B C\ar@{-->}[r]& A'\times_{B'} C'}$ is  the whisker map induced by the bottom homotopy pullback and 
$\xymatrix@1{A\bowtie_{B}C\ar@{.>}[r]& A'\bowtie_{B'}C'}$ is the whisker map induced by the top homotopy pushout. The rightern front vertical square is homotopy commutative by \lemref{whiskercube}.

\smallbreak
If the two squares we start with are homotopy pullbacks, then the Prism lemma and the Cube axiom imply that all vertical faces of the above diagram are also  homotopy pullbacks.
\end{proof}



\begin{proposition}\label{cofibrationproduit}
Assume that $\cM$ is pointed, and let $X \to A \to A'$ and $Y \to B \to B'$ be two homotopy cofibration sequences.
Consider the following join construction:
$$\xymatrix@R=1.5pc@C=1.5pc{
&A'\times B\ar[dr]\ar[rrrd]\\
A\times B\ar[rd]\ar[ur]&&J\ar@{-->}[rr]|-(.35)j&&A'\times B'\\
&A\times B'\ar[ru]\ar[rrru]
}$$
Then, there is a homotopy cofibration sequence:
$$\xymatrix{X \bowtie Y\ar[r]& J \ar[r]_(.35)j& A' \times B'}$$
\end{proposition}

This proposition relies on the Cube axiom; see \cite{Kah98} or \cite{CLOT03}, Proposition B.35.

\bibliographystyle{plain}

\bigskip

\noindent Jean-Paul Doeraene and Mohammed El Haouari

\noindent {\tt doeraene@math.univ-lille1.fr haouari@math.univ-lille1.fr}

\noindent {D\'epartement de Math\'ematiques\\
      UMR-CNRS 8524\\
         Universit\'e de Lille~1\\
         59655 Villeneuve d'Ascq Cedex\\
         France}

\end{document}